\documentclass{amsart}

\usepackage{amssymb,amsmath,amsthm,color}

\usepackage[bookmarksopen=true,bookmarksnumbered=true, bookmarksopenlevel=2, colorlinks=true,linkcolor=mblue,
citecolor=mblue,urlcolor=mblue, pdfstartview=FitH]{hyperref}

\definecolor{mblue}{rgb}{0,0,.8}

\usepackage[all]{xy}

\newcommand{\N}{\mathbb N}
\newcommand{\Z}{\mathbb Z}

\newcommand{\F}{\mathbb F}

\newcommand{\abcd}[4]{\left(\begin{smallmatrix}#1&#2\\#3&#4\end{smallmatrix}\right)}

\newtheorem{thm}{Theorem}
\newtheorem{lem}{Lemma}

\newtheorem{prop}{Proposition}
\newtheorem{cor}{Corollary}

 \DeclareMathOperator{\SL}{SL}  
   \DeclareMathOperator{\PSL}{PSL}

\def\dash---{\thinspace---\hskip.16667em\relax}

\begin{document}

\title{Lifts of projective congruence groups, II}
\author{Ian Kiming}
\address{Department of Mathematics, University of Copenhagen, Universitetsparken 5, DK-2100 Copenhagen \O ,
Denmark.}
\email{\href{mailto:kiming@math.ku.dk}{kiming@math.ku.dk}}

\begin{abstract} We continue and complete our previous paper `Lifts of projective congruence groups' \cite{noncongr} concerning the question of whether there exist noncongruence subgroups of $\SL_2(\Z)$ that are projectively equivalent to one of the groups $\Gamma_0(N)$ or $\Gamma_1(N)$. A complete answer to this question is obtained: In case of $\Gamma_0(N)$ such noncongruence subgroups exist precisely if $N\not\in \{ 3,4,8\}$ and we additionally have either that $4\mid N$ or that $N$ is divisible by an odd prime congruent to $3$ modulo $4$. In case of $\Gamma_1(N)$ these noncongruence subgroups exist precisely if $N>4$.

As in our previous paper the main motivation for this question is the fact that the above noncongruence subgroups represent a fairly accessible and explicitly constructible reservoir of examples of noncongruence subgroups of $\SL_2(\Z)$ that can serve as basis for experimentation with modular forms on noncongruence subgroups.
\end{abstract}

\maketitle
\section{Introduction.}\label{intro} Let $N\in\N$ and let $\Gamma$ be one of the standard congruence subgroups $\Gamma_0(N)$, $\Gamma_1(N)$, or $\Gamma(N)$. Denote by $\overline\Gamma$ the image of $\Gamma$ in $\PSL_2(\Z)$. For $\Gamma_1$ a subgroup of $\SL_2(\Z)$ (of finite index) we say that $\Gamma_1$ is a {\it lift} of $\overline\Gamma$ if $\Gamma_1$ projects to $\overline\Gamma$ under the canonical homomorphism $\SL_2(\Z) \rightarrow \PSL_2(\Z)$.

In our previous paper \cite{noncongr} we discovered that not only is it possible for a congruence $\overline\Gamma$ as above to have a {\it noncongruence lift}, i.e., a lift $\Gamma_1$ that is not a congruence subgroup, i.e., that does not contain $\Gamma(M)$ for any $M$, but that, in fact, the number of noncongruence lifts appear to usually dominate the number of congruence lifts. Here, `usually' should be taken to mean `apart from the cases where simple obstructions trivially prevent this, and apart from the cases where $N$ is small'. However, a number of hard cases were left out of the analysis in \cite{noncongr} and some of the results of that paper depended on machine computations.

The principal interest in these questions lies in the fact that noncongruence lifts of a group $\overline\Gamma$ provide relatively easy examples of noncongruence groups, and that, because our approach to these noncongruence lifts is constructive there is a possibility of studying spaces of modular forms on such noncongruence lifts. Cf.\ for example the analysis in \cite{noncongr} of spaces of modular forms of weight $3$ on the various lifts of the group $\overline{\Gamma_1(6)}$. As our knowledge of the arithmetic of modular forms on noncongruence subgroups is still fairly limited compared with the situation for congruence subgroups, having a readily accessible reservoir of examples of noncongruence subgroups is valuable as a tool for exploration and experimentation.

The purpose of the present paper is to augment the previous paper \cite{noncongr} so as to obtain a complete description of the situation for the above series of groups and for all $N$. We shall use a slightly different method of proof and are in fact able to prove everything from the ground up and also avoid all machine computations. Our results are as follows.

\begin{thm}\label{thm_congruence_lifts} Let $N\in\N$ and write $N = 2^s p_1^{s_1} \cdots p_t^{s_t}$ where the $p_i$ are distinct primes, $s_i\in\N$, and $s\in\Z_{\ge 0}$.

Then the number of congruence lifts of the subgroup $\overline{\Gamma_0(N)} \le \PSL_2(\Z)$ is
$$
\left\{ \begin{array}{ll} 1 ~, & \mbox{if $s\le 1$ and $p_i\equiv 1\pod{4}$ for all $i$} \\ 1+2^{s+t} ~, & \mbox{if $s\le 1$ and $p_i\equiv 3\pod{4}$ for at least one $i$} \\ 1+2^{\min \{ 3,s\} +t} ~, & \mbox{if $s\ge 2$} ~.\end{array} \right.
$$

The number of congruence lifts of the subgroup $\overline{\Gamma_1(N)} \le \PSL_2(\Z)$ is
$$
\left\{ \begin{array}{ll} 1 ~, & \mbox{if $N\le 2$} \\ 3 ~, & \mbox{if $N>2$ is odd} \\ 5 ~, & \mbox{if $N>2$ is even} ~.\end{array} \right.
$$

The number of congruence lifts of the subgroup $\overline{\Gamma(N)} \le \PSL_2(\Z)$ is
$$
\left\{ \begin{array}{ll} 1 ~, & \mbox{if $N=1$} \\ 3 ~, & \mbox{if $N>1$ is odd} \\ 5 ~, & \mbox{if $N=2$} \\ 9 ~, & \mbox{if $N>2$ is even} ~.\end{array} \right.
$$
\end{thm}

\begin{thm}\label{thm_lifts_gamma0(N)} Let $N\in\N$. All lifts of $\overline{\Gamma_0(N)} \le \PSL_2(\Z)$ are congruence subgroups of $\SL_2(\Z)$ if and only if either $N\in\{ 3,4,8\}$ or if $4\nmid N$ and all odd prime divisors of $N$ are congruent to $1$ modulo $4$.
\end{thm}

\begin{thm}\label{thm_lifts_gamma1(N)} Let $N\in\N$. All lifts of $\overline{\Gamma_1(N)} \le \PSL_2(\Z)$ are congruence subgroups of $\SL_2(\Z)$ if and only if $N\le 4$.
\end{thm}

Being complete classifications of the situation regarding the question of existence of noncongruence lifts of $\overline\Gamma$ when $\Gamma$ is either $\Gamma_0(N)$ or $\Gamma_1(N)$, the above theorems obviously deal with some of the remaining, hard cases that were left undecided in our previous paper \cite{noncongr}. These hard cases contain for instance all cases of $\Gamma_1(N)$ when $N$ is $4$ times an odd number $>1$, as well as the cases concerning $\Gamma_0(N)$ when $N$ is $4$ times an odd number $>1$ all of whose prime divisors are congruent to $1$ modulo $4$.

Below we shall prove the above theorems by a somewhat different, and more general method than in \cite{noncongr} although the basic principles of proof remain the same. Specifically, we shall utilise the paper \cite{kulkarni} for information about generators and relations for the group $\overline\Gamma$ in the various cases. Combined with the group theoretical analysis below in section \ref{congruence_lifts}, which leads to a proof of Theorem \ref{thm_congruence_lifts}, this allows us to prove Theorems \ref{thm_lifts_gamma0(N)} and \ref{thm_lifts_gamma1(N)} basically {\it ab initio}, i.e., essentially without referring back to the results of \cite{noncongr} (this is apart from certain elementary observations). In particular, this new approach bypasses the need for computation that was necessary in certain cases in the paper \cite{noncongr}.

The paper \cite{noncongr} already proved that there are noncongruence lifts of the group $\overline{\Gamma(N)}$ if (and only if) $N>2$. It would have been easy to include the proof of this fact here by using our current approach, but we have chosen not to do so.

\section{Congruence lifts}\label{congruence_lifts} Theorem \ref{thm_congruence_lifts} follows readily by combining Proposition \ref{prop_congruence_lifts_general} and Corollary \ref{cor_factor_groups} below.
\smallskip

For a finite group $G$ denote as usual by $G'$ the derived group of $G$, and by $G^2$ the characteristic subgroup of $G$ generated by all squares of elements of $G$. The subgroup $G'G^2$ is then the largest normal subgroup of $G$ with elementary abelian quotient of $2$-power order. (One has in fact $G'G^2 = G^2$, since all commutators can be expressed as products of squares. Also, if $G$ is a $2$-group, $G'G^2$ is the Frattini subgroup of $G$.  However, we will not need this information).

Let us recall Wohlfahrt's notion of the `general level' of a subgroup $\Gamma$ of finite index in $\SL_2(\Z)$, cf.\ \cite{wohlfahrt}: the general level of $\Gamma$ is defined to be the least common multiple of all cusp widths where these widths are computed relative to the projective image $\overline\Gamma \le \PSL_2(\Z)$. I.e., the width of a cusp $c$ is the least $n\in\N$ such that
$$
\pm g \abcd{1}{n}{0}{1} g^{-1} \in \Gamma
$$
where $g\in\SL_2(\Z)$ is such that $g\infty = c$. Thus, the general level of $\Gamma$ depends only on its projective image $\overline\Gamma$.

\begin{prop}\label{prop_congruence_lifts_general} Let $\Gamma$ be a congruence subgroup of $\SL_2(\Z)$ of general level $N$.

Then $\Gamma \ge \Gamma(2N)$. We put $G := \Gamma / \Gamma(2N)$ and consider the quotient $G/G'G^2$ as a vector space over $\F_2$. Define:
$$
d : = \dim_{\F_2} G/G'G^2 ~.
$$

Abusing notation and denoting by $-1$ both the matrix $\abcd{-1}{0}{0}{-1} \in \SL_2(\Z)$ as well as its image in $G$, the number of congruence lifts of $\overline\Gamma$ is:
$$
\left\{ \begin{array}{ll} 1~,& \mbox{if $-1\in G'G^2$} \\ 1 + 2^{d-1} ~,& \mbox{if $-1\in\Gamma$ but $-1\not\in G'G^2$} \\
1 + 2^d ~,& \mbox{if $-1\not\in\Gamma$}~.
\end{array} \right.
$$
\end{prop}

Before the proof, we need the following simple lemma from linear algebra.

\begin{lem}\label{lemma_vector_spaces} Let $p$ be a prime, let $d\in\N$ and let $V$ be a vector space of dimension $d$ over $\F_p$.
\smallskip

If $0\neq v\in V$ is a given nonzero vector then the number of subspaces of $V$ of codimension $1$ and not containing $v$ is $p^{d-1}$.
\end{lem}
\begin{proof} The statement is obviously true when $d=1$, so we proceed under the assumption that $d>1$.

The subspaces of $V$ of codimension $1$ are in $1-1$ correspondence with surjective linear maps $V\rightarrow \F_p$ modulo scalars $\neq 0$. Hence the number $A$ of such subspaces is:
$$
A = \frac{p^d-1}{p-1} = p^{d-1} + \ldots + 1
$$
as the cardinality of the dual vector space $V^{\ast}$ is $\# V^{\ast} = \# V = p^d$.
\smallskip

Now let a nonzero vector $v\in V$ be given. The number $B$ of codimension $1$ subspaces of $V$ not containing $v$ is $B=A-C$ where $C$ is the number of codimension $1$ subspaces $W$ of $V$ with $v\in W$.

Clearly, $C$ is the number of surjective linear maps $V\rightarrow \F_p$ with $v$ in the kernel, counted modulo scalars $\neq 0$. Again, this can be identified as the number of surjective linear maps $V/\langle v \rangle \rightarrow \F_p$ modulo scalars $\neq 0$. As $V/\langle v \rangle$ has dimension $d-1$, it follows from the first part of the proof that
$$
C = \frac{p^{d-1}-1}{p-1}
$$
But then:
$$
B = A - C = p^{d-1} ~.
$$
\end{proof}

\begin{proof}[Proof of Proposition \ref{prop_congruence_lifts_general}] As $\Gamma$ is a congruence subgroup of $\SL_2(\Z)$ of general level $N$, we know by Proposition $3$ of \cite{noncongr} that $\Gamma$ contains $\Gamma(2N)$. (This result is essentially due to Wohlfahrt \cite{wohlfahrt}, cf.\ the discussion after Proposition $1$ of \cite{noncongr}). Furthermore, if $\Gamma_1$ is a congruence lift of $\overline\Gamma$ then $\Gamma_1$ also has general level $N$ and hence also contains $\Gamma(2N)$. Also, for any lift $\Gamma_1$ of $\overline\Gamma$ we must necessarily have $\langle \Gamma_1 , -1 \rangle = \langle \Gamma , -1 \rangle$ since $\langle -1 \rangle$ is the kernel of the projection map $\SL_2(\Z) \rightarrow \PSL_2(\Z)$.
\smallskip

Suppose first that $-1\in\Gamma$. It then follows from the above that congruence lifts of $\overline\Gamma$ correspond one-to-one to subgroups $U$ of $G := \Gamma / \Gamma(2N)$ such that $G$ is generated by $U$ and the image of $-1$. Such a subgroup $U$ is then necessarily of index $\le 2$ in $G$ and hence is normal in $G$ and contains $G'G^2$. Thus, these subgroups $U$ are in one-to-one correspondence with subspaces $W$ of the $\F_2$-vector space $V := G/G'G^2$ such that $V$ is generated by $W$ and the image $v$ of $-1$. If $v=0$, i.e., if $-1\in G'G^2$ there is only one such subspace, namely $V$ itself. On the other hand, if $-1\not\in G'G^2$ then $v\neq 0$, and apart from $V$ itself the possibilities for the subspaces $W$ are the codimension $1$ subspaces of $V$ not containing $v$. By lemma \ref{lemma_vector_spaces} the number of the latter subspaces is $2^{d-1}$ where $d := \dim _{\F_2} G/G'G^2$; observe that, because of the condition $-1\not\in G'G^2$, we certainly have $d\ge 1$ in this case.

We have now established the proposition in case $-1\in\Gamma$.
\smallskip

Suppose then that $-1\not\in\Gamma$ and consider the group $\widetilde\Gamma := \langle \Gamma , -1 \rangle$. Then $\widetilde\Gamma$ is a congruence subgroup with image $\overline\Gamma$ in $\PSL_2(\Z)$ and hence also has general level $N$. As now $-1\in\widetilde\Gamma$, and as congruence lifts of $\overline\Gamma$ and of $\overline{\widetilde\Gamma}$ are trivially the same things, we can now apply the previous discussion to $\widetilde\Gamma$. So, we consider the group $\widetilde G := \widetilde\Gamma / \Gamma(2N)$ and the dimension $\tilde d := \dim_{\F_2} \widetilde G / \widetilde G' \widetilde G^2$. Clearly, $\widetilde G \cong \langle -1 \rangle \times G$ whence we see that $\tilde d = d+1$. As $\widetilde G' \widetilde G^2 = G'G^2$, the hypothesis $-1\not\in\Gamma$ implies $-1\not\in \widetilde G' \widetilde G^2$. Hence the previous discussion implies the statement of the proposition in the present case.
\end{proof}

\begin{prop}\label{prop_factor_groups}
\smallskip

\noindent (i) Let $p$ be an odd prime and let $s\in\N$.

If $G_0:=\Gamma_0(p^s)/\Gamma(p^s)$ then $G_0/(G_0'G_0^2)$ is cyclic of order $2$. Further, $-1\in G_0'G_0^2$ if and only if $p\equiv 1\pod{4}$.

If $G_1:=\Gamma_1(p^s)/\Gamma(p^s)$ then $G_1 = G_1'G_1^2$.

We have $-1\in \Gamma_0(p^s)$, but $-1\not\in \Gamma_1(p^s), \Gamma(p^s)$
\smallskip

\noindent (ii) Let $s\in\Z_{\ge 0}$.

If $G_0 := \Gamma_0(2^s)/\Gamma(2^{s+1})$ then:
$$
G_0/(G_0'G_0^2) \cong \left\{ \begin{array}{ll} (\Z/2)^{s+1} & \mbox{if $s\le 2$} \\ (\Z/2)^4 & \mbox{if $s\ge 3$}. \end{array} \right.
$$

We have $-1\in \Gamma_0(2^s)$, but $-1\in G_0'G_0^2$ if and only if $s \le 1$.

If $G_1 := \Gamma_1(2^s)/\Gamma(2^{s+1})$ then:
$$
G_1/(G_1'G_1^2) \cong \left\{ \begin{array}{ll} \Z/2 & \mbox{if $s=0$} \\ (\Z/2)^2 & \mbox{if $s\ge 1$}. \end{array} \right.
$$

We have $-1\in \Gamma_1(2^s)$ if and only if $-1\in G_1'G_1^2$ if and only if $s\le 1$.

Finally, if $G := \Gamma(2^s)/\Gamma(2^{s+1})$ then:
$$
G/(G'G^2) \cong \left\{ \begin{array}{ll} \Z/2 & \mbox{if $s=0$} \\ (\Z/2)^3 & \mbox{if $s\ge 1$} \end{array} \right.
$$
and we have $-1\in \Gamma(2^s)$ if and only if $s\le 1$, and $-1\in G'G^2$ if and only if $s=0$.
\end{prop}
\begin{proof} For $N\in\N$ we have in general that $\SL_2(\Z)/\Gamma(N) \stackrel{\sim}{\longrightarrow} \SL_2(\Z/N)$ with the isomorphism induced by reducing matrix entries modulo $N$. Also, we have in $\Gamma_0(N)$ a normal series $\Gamma_0(N) \unrhd \Gamma_1(N) \unrhd \Gamma(N)$ with successive quotients:
$$
\Gamma_0(N)/\Gamma_1(N) \cong (\Z/N)^{\times}
$$
generated by matrices modulo $N$ of the shape $\abcd{a}{0}{0}{b}$ where $a$ and $b$ are integers with $ab\equiv 1\pod{N}$, as well as
$$
\Gamma_1(N)/\Gamma(N) \cong \Z/N
$$
generated by the matrix $\abcd{1}{1}{0}{1}$ modulo $N$.
\medskip

\noindent (i) As $p$ is odd, the group $(\Z/p^s)^{\times}$ is cyclic of order $p^{s-1} (p-1)$. If $a$ is a generator and $b$ is such that $ab\equiv 1\pod{p^s}$, the quotient $\Gamma_0(p^s)/\Gamma_1(p^s)$ is then generated by the image of
$$
\xi := \abcd{a}{0}{0}{b} \pmod{p^s} .
$$
On the other hand, $\Gamma_1(p^s)/\Gamma(p^s)$ is generated modulo $p^s$ by:
$$
\tau := \abcd{1}{1}{0}{1} ~.
$$

Now, since $p$ is odd we see that $\tau$ is a suitable power of $\tau^2$ so that $\langle \xi^2 , \tau^2 \rangle = \langle \xi^2 , \tau \rangle$. But $\langle \xi^2 , \tau \rangle$ is normal in $G_0$ with quotient cyclic of order $2$. It follows that $G_0'G_0^2 = \langle \xi^2 , \tau^2 \rangle = \langle \xi^2 , \tau \rangle$ and hence that $G_0/(G_0'G_0^2)$ is cyclic of order $2$, as claimed. We also see that $-1\in G_0'G_0^2 = \langle \xi^2 , \tau \rangle$ if and only if $-1$ is a square modulo $p$ which happens precisely if $p\equiv 1 \pod{4}$.
\smallskip

As $s\ge 1$ and $p$ is odd, the group $\Gamma_1(p^s)$ does not contain $-1$. Since $G_1$ is cyclic of odd order $p^s$ we have $G_1 = G_1'G_1^2$.
\medskip

\noindent (ii) We will be considering the normal series:
$$
\Gamma_0(2^s) \unrhd \Gamma_1(2^s) \unrhd \Gamma(2^s) \unrhd \Gamma(2^{s+1})
$$
as well as the corresponding normal series in $G_0 = \Gamma_0(2^s)/\Gamma(2^{s+1})$.
\smallskip

If $s=0$ the three groups $G_0$, $G_1$, and $G$ coincide and are isomorphic to $\SL_2(\Z) / \Gamma(2) \cong \SL_2(\F_2) \cong S_3$. It follows that $G/(G'G^2)$ is cyclic of order $2$ in this case. We have $-1\in\Gamma(2)$ and hence the image of this element is trivial in $G$.
\smallskip

We may now assume $s\ge 1$ for the rest of the proof.

Then the quotient $G:=\Gamma(2^s) / \Gamma(2^{s+1})$ has order $8$, and one checks that it is in fact isomorphic to $(\Z/2)^3$ and generated by the following matrices modulo $2^{s+1}$:
$$
\alpha := \abcd{1}{2^s}{0}{1} ~,~~ \beta := \abcd{1+2^s}{2^s}{0}{1+2^s} ~,~~ \gamma := \abcd{1}{0}{2^s}{1} ~.
$$

It follows that $G'G^2 = 1$ and that $-1\not\in G'G^2$.
\smallskip

The quotient $\Gamma_1(2^s) / \Gamma(2^s)$ is generated by the image of the matrix:
$$
\tau := \abcd{1}{1}{0}{1}
$$
so that $\tau^{2^s} = \alpha$ in $G$. One computes that $\tau\alpha = \alpha\tau$, $\tau\beta = \beta\tau$, and that:
$$
\tau\gamma\tau^{-1} = \beta\gamma ~.
$$

We see that $G_1/\langle \tau^2,\beta \rangle$ is isomorphic to $(\Z/2)^2$, generated by the images of $\tau$ and $\gamma$. Hence this quotient is an elementary abelian $2$-group whence $\langle \tau^2,\beta \rangle \ge G_1'G_1^2$. On the other hand, $\beta\in G_1'$ by the above. Hence $G_1'G_1^2 = \langle \tau^2,\beta \rangle$. We have $-1\in \Gamma_1(2^s)$ if and only if $s\le 1$. When $s=1$ we do in fact have $-1\in G_1'G_1^2$ as this element in this case coincides with $\alpha\beta = \tau^2 \beta$.
\smallskip

It now remains to deal with the group $G_0$ for $s\ge 1$.

Assume first that in fact $s\ge 3$. Then $\Gamma_0(2^s) / \Gamma_1(2^s) \cong (\Z/2^s)^{\times} \cong \Z/2 \times \Z/2^{s-2}$ generated by the images of the following matrices modulo $2^{s+1}$
$$
\abcd{-1}{0}{0}{-1} \quad \mbox{and} \quad \xi := \abcd{a}{0}{0}{b}
$$
where we have chosen $a$ such that $a$ has order $2^{s-2}$ in $(\Z/2^s)^{\times}$, generating the second factor in the above decomposition (for instance, we may choose $a:=5$), and $b$ such that
$$
ab\equiv 1 \pod{2^{s+1}} ~.
$$

One then checks that $\xi$ commutes with $\alpha$, $\beta$, $\gamma$ above, and that:
$$
[\xi,\tau] := \xi\tau\xi^{-1}\tau^{-1} = \tau^{a^2-1} \in \langle \tau^2 \rangle ~.
$$

It can then be concluded that $G_0/\langle \xi^2 , \tau^2 , \beta \rangle$ is isomorphic to $(\Z/2)^4$ with the quotient generated by the images of $-1$, $\xi$, $\tau$, $\gamma$ (recall that $\tau^{2^s} = \alpha$). In particular, this quotient is an elementary abelian $2$-group, and so (as $\beta$ is a commutator) we can deduce that $G_0'G_0^2 = \langle \xi^2 , \tau^2 , \beta \rangle$ and that $-1\not\in G_0'G_0^2$.
\smallskip

If $s=2$ the difference with the case $s\ge 3$ is only that $\Gamma_0(2^s) / \Gamma_1(2^s) \cong (\Z/2^s)^{\times}$ is now cyclic of order $2$, generated by $-1$. The conclusions in this case then follow in the same way as in the case $s\ge 3$.
\smallskip

When $s=1$ the groups $G_0$ and $G_1$ coincide so that we have already discussed this case. In particular, we have $-1\in G_0'G_0^2$ when $s=1$.
\end{proof}

\begin{cor}\label{cor_factor_groups} Let $N\in\N$ and write $N = 2^s p_1^{s_1} \cdots p_t^{s_t}$ where the $p_i$ are distinct primes, $s_i\in\N$, and $s\in\Z_{\ge 0}$.
\smallskip

If $G_0$ denotes the group $G_0 := \Gamma_0(N) / \Gamma(2N)$ then
$$
G_0/(G_0'G_0^2) \cong (\Z/2)^{\min\{ 4,s+1\} +t} ~.
$$

We have $-1\in \Gamma_0(N)$, but $-1\in G_0'G_0^2$ if and only if $s\le 1$ and $p_i\equiv 1\pod{4}$ for all $i$.
\smallskip

If $G_1$ denotes the group $G_1 := \Gamma_1(N) / \Gamma(2N)$ then
$$
G_1/(G_1'G_1^2) \cong (\Z/2)^{\min\{ 2,s+1\}} ~.
$$

We have $-1\in \Gamma_1(N)$ if and only if $-1\in G_1'G_1^2$ if and only if $s\le 1$ and $t=0$, i.e., if and only if $N\le 2$.
\smallskip

If $G$ denotes the group $G := \Gamma(N) / \Gamma(2N)$ then
$$
G/(G'G^2) \cong \left\{ \begin{array}{ll} \Z/2 ~,& \mbox{if $s=0$} \\ (\Z/2)^3 ~,& \mbox{if $s\ge 1$} ~.\end{array} \right.
$$

We have $-1\in \Gamma(N)$ if and only if $s\le 1$ and $t=0$, i.e., if and only if $N\le 2$. Furthermore, $-1\in G'G^2$ if and only if $s=0$.
\end{cor}
\begin{proof} We have a natural isomorphism:
$$
\SL_2(\Z)/\Gamma(2N) \cong \SL_2(\Z/(2N)) \cong \SL_2(\Z/2^{s+1}) \times \SL_2(\Z/p_1^{s_1}) \times \cdots \times \SL_2(\Z/p_t^{s_t})
$$
given in concrete terms as
$$
A \mapsto (A \pmod{2N}) \leftrightarrow ((A \pmod{2^{s+1}}) , (A \pmod{p_1^{s_1}}) , \ldots , (A \pmod{p_t^{s_t}}) )
$$
for matrices $A\in\SL_2(\Z)$ (cf.\ for instance Lemma 4.2.3 of \cite{miyake}).

Under this isomorphism, the subgroup $\Gamma_0(N)/\Gamma(2N) \le \SL_2(\Z)/\Gamma(2N)$ clearly injects into the subgroup
$$
\Gamma_0(2^s)/\Gamma(2^{s+1}) \times \Gamma_0(p_1^{s_1})/\Gamma(p_1^{s_1}) \times \cdots \times \Gamma_0(p_t^{s_t})/\Gamma(p_t^{s_t})
$$
of
\begin{eqnarray*} && \SL_2(\Z)/\Gamma(2^{s+1}) \times \SL_2(\Z)/\Gamma(p_1^{s_1}) \times \cdots \times \SL_2(\Z)/\Gamma(p_t^{s_t}) \\
& \cong & \SL_2(\Z/2^{s+1}) \times \SL_2(\Z/p_1^{s_1}) \times \cdots \times \SL_2(\Z/p_t^{s_t}) ~.
\end{eqnarray*}

That this map is surjective and hence an isomorphism follows for instance by comparison of orders.
\smallskip

We observe that a similar remark holds for the quotients $\Gamma_1(N)/\Gamma(2N)$ and $\Gamma(N)/\Gamma(2N)$, and that these decompositions are obviously compatible with formation of the characteristic subgroups $G'G^2$ etc. Thus the claims of the corollary are seen to follow immediately from Proposition \ref{prop_factor_groups}.
\end{proof}

\section{Proofs of Theorems \ref{thm_lifts_gamma0(N)} and \ref{thm_lifts_gamma1(N)}}\label{proofs} Denote by $\overline{\Gamma}$ either one of the groups $\overline{\Gamma_0(N)}$ and $\overline{\Gamma_1(N)}$. The proofs follow the same general strategy as in \cite{noncongr}: Using information about a presentation of $\overline{\Gamma}$ in terms of generators and relations we obtain via Lemma $4$ of \cite{noncongr} the total number of lifts of $\overline{\Gamma}$ to $\SL_2(\Z)$. Comparing this with the information given by Theorem \ref{thm_congruence_lifts} above we can decide whether all lifts are congruence subgroups.

As far as a presentation of $\overline{\Gamma_0(N)}$ is concerned, the paper \cite{noncongr} cited the results of Chuman in \cite{chuman}. It has since come to our attention that Chuman's paper in fact contains errors, cf.\ \cite{orive}. However, the derivation of the results of \cite{noncongr} did not depend in any way on Chuman's paper.
\smallskip

We shall base our discussion here on Kulkarni's paper \cite{kulkarni}. Let us recall some consequences of the principal results of that paper: First, the paper describes the group $\overline{\Gamma}$ (in fact, any subgroup of finite index in $\PSL_2(\Z)$) by certain combinatorial objects called generalised Farey sequences. We will not describe these here except to say that such a sequence contains certain numbers $a$, $b$, and $r$ of `even', `odd', and `pairs of free' `intervals', respectively, that can be used to display $\overline{\Gamma}$ as given in terms of $a+b+r$ generators, cf.\ Theorem 6.1 of \cite{kulkarni}. These generators are such that $a$ of them have order $2$, $b$ have order $3$, the remaining $r$ are of infinite order, and there are no further relations between the generators. In other words, a presentation of $\overline{\Gamma}$ is given in terms of generators $\bar{g}_1,\ldots,\bar{g}_{a+b+r}$ with the following relations:
$$
\bar{g}_1^2 = \ldots = \bar{g}_a^2 = 1, ~~ \bar{g}_{a+1}^3 = \ldots = \bar{g}_{a+b}^3 = 1.
$$

The numbers $a$, $b$, and $r$ are determined as follows. We have $a=e_2$, $b=e_3$, the number of conjugacy classes of subgroups of $\overline{\Gamma}$ of order $2$ and $3$, respectively, cf.\ sections (7.1), (7.2) of \cite{kulkarni}. Furthermore, the number $r$ is given as
$$
r = \frac{1}{6} (d - 3e_2 -4e_3) + 1,
$$
where $d$ is the index $[\PSL_2(\Z):\overline{\Gamma}]$; cf.\ equation (7.1.2) of \cite{kulkarni}.

\subsection{Proof of Theorem \ref{thm_lifts_gamma0(N)}}\label{gamma_0} Specialising the above to the case $\overline{\Gamma} = \overline{\Gamma_0(N)}$, we have
$$
d = N \prod_{p\mid N} \left( 1 + \frac{1}{p} \right)
$$
and the determination of the numbers $e_2$ and $e_3$ is well-known, cf.\ \cite{miyake}, \S 4.2, for instance.

The cases left undecided by Theorem $2$ of \cite{noncongr} are those where $N$ is $3$, $4$, or $8$ times an odd number greater than $1$ all of whose prime divisors are congruent to $1$ modulo $4$. Thus, we could limit ourselves to discussing these remaining cases.

However, part of the proof of Theorem $2$ of \cite{noncongr} depended on (machine) computations and we want to show here that these can all be avoided. Hence, we will use only the following two results from \cite{noncongr}: First, if $p$ is a prime greater than $3$, but congruent to $3$ modulo $4$ then there exist noncongruence lifts of $\overline{\Gamma_0(p)}$; this is in contrast with the situation for $\overline{\Gamma_0(3)}$ that has precisely $3$ lifts all of which are congruence. Cf.\ Lemma $30$ of \cite{noncongr}. The proof utilized Rademacher's presentation of $\overline{\Gamma_0(p)}$ for $p$ prime as given in \cite{rademacher}.

Secondly, if $4\nmid N$ and all odd prime divisors of $N$ are congruent to $1$ modulo $4$ then all lifts of $\overline{\Gamma_0(N)}$ are congruence, cf.\ part $(i)$ of Theorem 2 of \cite{noncongr}. The proof consists of a simple observation that in this case, any lift necessarily contains $-1$ and hence actually equals $\Gamma_0(N)$.
\smallskip

We will also utilize the following simple observation (cf.\ Lemma $5$ of \cite{noncongr}): 

\begin{lem}\label{lem:subgroups} Suppose that $\Gamma_1$ and $\Gamma_2$ are subgroups of $\SL_2(\Z)$ with $\Gamma_2 \subseteq \Gamma_1$.

If there exists a noncongruence lift of $\overline{\Gamma_1}$ then $\overline{\Gamma_2}$ also has a noncongruence lift.
\end{lem}

The lemma follows since the pre-image of $\overline{\Gamma_2}$ inside a noncongruence lift of $\overline{\Gamma_1}$ must obviously necessarily be a noncongruence subgroup of $\SL_2(\Z)$.
\smallskip

Using Lemma \ref{lem:subgroups} together with the starting points described above, one checks that in order to prove Theorem \ref{thm_lifts_gamma0(N)} it suffices to show:
\begin{itemize} \item If $N \in \{ 4, 8\}$ all lifts of $\overline{\Gamma_0(N)}$ are congruence,
\item If $N \in \{ 6, 9, 16\}$ there are noncongruence lifts of $\overline{\Gamma_0(N)}$,
\item $N$ is $3$ or $4$ times an odd number $>1$ all of whose prime divisors are \\ congruent to $1$ modulo $4$, there are noncongruence lifts of $\overline{\Gamma_0(N)}$.\end{itemize}
\smallskip

When $N=4,6,8,9,16$ one finds $e_2=e_3=0$ and $d=6,12,12,12,24$, respectively, so that $\overline{\Gamma_0(N)}$ is generated in these cases by $r=2,3,3,3,5$ elements with no relations, respectively. Thus, the total number of lifts of $\overline{\Gamma_0(N)}$ not containing $-1$ is $2^2,2^3,2^3,2^3,2^5$, respectively, cf.\ Lemma 4 of \cite{noncongr}. On the other hand, by Theorem \ref{thm_congruence_lifts} the number of congruence lifts of $\overline{\Gamma_0(N)}$ not containing $-1$ is $2^2,2^2,2^3,2,2^3$, respectively for these cases. It follows that when $N\in\{ 4,8\}$ all lifts of $\overline{\Gamma_0(N)}$ are congruence whereas there are noncongruence lifts when $N\in\{ 6,9,16\}$.
\smallskip

Next suppose that $N=4\cdot p_1^{s_1} \cdots p_t^{s_t}$ where $t\ge 1$ and the distinct prime divisors $p_i$ are all congruent to $1$ modulo $4$. In this case, $e_2=e_3=0$, and
$$
d = 6\cdot \prod_{i=1}^t p_i^{s_i-1}(p_i+1)
$$
whereas the number of congruence lifts of $\overline{\Gamma_0(N)}$ not containing $-1$ in this case is $2^{2+t}$, by Theorem \ref{thm_congruence_lifts}. But the total number of lifts not containing $-1$ is $2^{d/6+1}$. Since $d/6+1 \ge 1+6^t$ which is certainly greater than $2+t$ we conclude that there are in fact noncongruence lifts of $\overline{\Gamma_0(N)}$ in this case.
\smallskip

Suppose then that $N=3\cdot p_1^{s_1} \cdots p_t^{s_t}$ where $t\ge 1$ and the distinct prime divisors $p_i$ are all congruent to $1$ modulo $4$, and that, additionally, $p_i\equiv -1 \pod{3}$ for at least one $i$. Then $e_2=e_3=0$,
$$
d = 4\cdot \prod_{i=1}^t p_i^{s_i-1}(p_i+1) ,
$$
and the total number of lifts of $\overline{\Gamma_0(N)}$ not containing $-1$ is $2^{d/6+1}$. As the number of congruence lifts not containing $-1$ is $2^t$ by Theorem \ref{thm_congruence_lifts}, one verifies again that there are noncongruence lifts of $\overline{\Gamma_0(N)}$ in this case.
\smallskip

Suppose then finally that $N=3\cdot p_1^{s_1} \cdots p_t^{s_t}$ where $t\ge 1$ and the distinct prime divisors $p_i$ are all congruent to $1$ modulo $12$. Then $e_2=0$, but $e_3=2^t$. Again,
$$
d = 4\cdot \prod_{i=1}^t p_i^{s_i-1}(p_i+1) ,
$$
but now
$$
r = 1 + \frac{2}{3} \left( -2^t + \prod_{i=1}^t p_i^{s_i-1}(p_i+1) \right) .
$$

Again the number of congruence lifts not containing $-1$ is $2^t$ and the question is whether $2^r > 2^t$. An elementary computation shows this to be the case, and hence the conclusion is again that there are noncongruence lifts of $\overline{\Gamma_0(N)}$ in this case.

\subsection{Proof of Theorem \ref{thm_lifts_gamma1(N)}}\label{gamma_1} If $N\le 3$ then $\overline{\Gamma_1(N)} = \overline{\Gamma_0(N)}$ and hence (by the already proved Theorem \ref{thm_lifts_gamma0(N)}) all lifts of $\overline{\Gamma_1(N)}$ are congruence.

Assume then $N\ge 4$ from now on. One then has $e_2=e_3=0$. Further, the index $d=[\PSL_2(\Z):\overline{\Gamma_1(N)}]$ is
$$
d = \frac{N^2}{2} \prod_{p\mid N} \left( 1 - \frac{1}{p^2} \right) ,
$$
and it follows that $\overline{\Gamma_1(N)}$ is generated by
$$
r = 1 + \frac{N^2}{12} \prod_{p\mid N} \left( 1 - \frac{1}{p^2} \right)
$$
elements with no relations.

Thus, the total number of lifts of $\overline{\Gamma_1(N)}$ is $1+2^r$. The question then becomes whether this number exceeds the number of congruence lifts as given by Theorem \ref{thm_congruence_lifts}.

When $N=4$ we have $r=2$ and so $1+2^r = 5$ which by Theorem \ref{thm_congruence_lifts} is precisely the number of congruence lifts. Hence all lifts are congruence when $N=4$.

Suppose then that $N\ge 5$. Then $N$ is divisible by either $6$, $8$, $9$, or a prime $p\ge 5$.

Now, when $N$ equals $6$, $8$, $9$, or a prime $p\ge 5$, we find that $r$ is $3$, $5$, $7$ or $\frac{p^2+11}{12}$, respectively, and so we see in each case that $1+2^r$ exceeds the number of congruence lifts given by Theorem \ref{thm_congruence_lifts}. Thus there are noncongruence lifts in each of these cases.

By Lemma \ref{lem:subgroups} we can then conclude the existence of noncongruence lifts of $\overline{\Gamma_1(N)}$ whenever $N$ is divisible by $6$, $8$, $9$, or by a prime $p\ge 5$.

Theorem \ref{thm_lifts_gamma1(N)} is proved.


\end{document}